\pdfoutput=1
\RequirePackage{ifpdf}
\ifpdf 
\documentclass[pdftex]{sigma}
\else
\documentclass{sigma}
\fi

\begin{document}

\allowdisplaybreaks

\renewcommand{\thefootnote}{$\star$}

\newcommand{\arXivNumber}{1511.06057}

\renewcommand{\PaperNumber}{044}

\FirstPageHeading

\ShortArticleName{Polynomial Sequences Associated with the Moments of Hypergeometric Weights}

\ArticleName{Polynomial Sequences Associated with the Moments\\ of Hypergeometric Weights\footnote{This paper is a~contribution to the Special Issue
on Orthogonal Polynomials, Special Functions and Applications.
The full collection is available at \href{http://www.emis.de/journals/SIGMA/OPSFA2015.html}{http://www.emis.de/journals/SIGMA/OPSFA2015.html}}}

\Author{Diego DOMINICI}

\AuthorNameForHeading{D.~Dominici}

\Address{Department of Mathematics, State University of New York at New Paltz,\\ 1 Hawk Dr., New Paltz, NY 12561-2443, USA}
\Email{\href{mailto:dominicd@newpaltz.edu}{dominicd@newpaltz.edu}}
\URLaddress{\url{https://www2.newpaltz.edu/~dominicd/}}

\ArticleDates{Received November 23, 2015, in f\/inal form April 25, 2016; Published online April 29, 2016}

\Abstract{We present some families of polynomials related to the moments of weight
functions of hypergeometric type. We also consider dif\/ferent types of
generating functions, and give several examples.}


\Classification{44A60; 33C20; 05A15}

\rightline{\it Dedicated to my daughter Malena, luz de mi vida!}

\renewcommand{\thefootnote}{\arabic{footnote}}
\setcounter{footnote}{0}

\section{Introduction}

Let $\{\mu_{n}\}$ be a sequence of complex numbers and
$\mathcal{L}\colon \mathbb{C}[x] \rightarrow\mathbb{C}$ be a linear
functional def\/ined by
\begin{gather*}
\mathcal{L}\big( x^{n}\big) =\mu_{n},\qquad n=0,1,\ldots.
\end{gather*}
Then, $\mathcal{L}$ is called the \emph{moment functional} determined by the
formal moment sequence $\{\mu_{n}\}$. The number~$\mu_{n}$ is
called the \emph{moment} of order~$n$. The task of f\/inding an explicit
representation for the functional $\mathcal{L}$ is called a \emph{moment
problem} \cite{MR0184042,MR0458081,MR0008438}.

A sequence $\{\Pi_{n}(x)\} \subset\mathbb{C}[x]$, with $\deg( \Pi_{n}) =n$ is called an
\emph{orthogonal polynomial sequence} with respect to~$\mathcal{L}$ provided
that~\cite{MR0481884}
\begin{gather*}
\mathcal{L} ( \Pi_{n}\Pi_{m} ) =K_{n}\delta_{n,m},\qquad
n,m=0,1,\ldots,
\end{gather*}
where $K_{n}\neq0$ and $\delta_{n,m}$ is Kronecker's delta.

The moments play a fundamental role in the theory of orthogonal polynomials
since, among other results, we have the determinantal representation
\begin{gather*}
\Pi_{n}(x) =C_{n}\left\vert
\begin{matrix}
\mu_{0} & \mu_{1} & \cdots & \mu_{n}\\
\mu_{1} & \mu_{2} & \cdots & \mu_{n+1}\\
\vdots & \vdots & \ddots & \vdots\\
\mu_{n-1} & \mu_{n} & \cdots & \mu_{2n-1}\\
1 & x & \cdots & x^{n}
\end{matrix}
\right\vert ,
\end{gather*}
for some normalization constant $C_{n}\neq0$, with
\begin{gather*}
\left\vert
\begin{matrix}
\mu_{0} & \mu_{1} & \cdots & \mu_{n}\\
\mu_{1} & \mu_{2} & \cdots & \mu_{n+1}\\
\vdots & \vdots & \ddots & \vdots\\
\mu_{n} & \mu_{n+1} & \cdots & \mu_{2n}
\end{matrix}
\right\vert \neq0,\qquad n=0,1,\ldots.
\end{gather*}
Given their importance, it is very striking that they are not explicitly
listed in the standard books on orthogonal polynomials, or even in
encyclopedic texts such as~\cite{MR2656096}. In fact, the only place where we
found a comprehensive enumeration of the moments of classical orthogonal
polynomials was the recent article~\cite{MR3286553}, based on the results
obtained in the Ph.D.\ Thesis of the f\/irst author~\cite{Njionou}\footnote{We
thank one of the referees for stressing this fact.}.

In this paper, we focus our attention on linear functionals def\/ined by
\begin{gather}
\mathcal{L}(f) =
{\displaystyle\sum\limits_{x=0}^{\infty}}
f(x)\rho\big( x;\vec{\alpha},\vec{\beta},c\big) ,
\label{L}
\end{gather}
where the weight function $\rho( x;\vec{\alpha}
,\vec{\beta},c) $ is of the form
\begin{gather}
\rho\big( x;\vec{\alpha},\vec{\beta},c\big)
=\frac{( \vec{\alpha})_{x}}{(
\vec{\beta}+1)_{x}}\frac{c^{x}}{x!}, \label{weight}
\end{gather}
with
\begin{gather*}
 ( \vec{\alpha} )_{x} = ( \alpha_{1})
_{x} ( \alpha_{2} )_{x}\cdots ( \alpha_{p} )_{x},\qquad
 ( \vec{\beta}+1 )_{x} = ( \beta_{1}+1 )
_{x} ( \beta_{2}+1 )_{x}\cdots ( \beta_{q}+1 )_{x},
\end{gather*}
and $(a)_{x}$ denotes the Pochhammer symbol (also called
shifted or rising factorial) def\/ined by~\cite[(5.2.4)]{MR2723248}
\begin{gather*}
(a)_{0} =1, \qquad (a)_{x} =a ( a+1 ) \cdots ( a+x-1 ),\qquad x\in\mathbb{N},
\end{gather*}
or by
\begin{gather*}
(a)_{x}=\frac{\Gamma ( a+x ) }{\Gamma (a ) },\qquad a+x\neq0,-1,\ldots,
\end{gather*}
and $\Gamma(z) $ is the Gamma function. Unless stated otherwise,
we always assume that
\begin{gather*}
\beta_{i}>-1,\qquad 1\leq i\leq q,
\end{gather*}
and we will use the notation
\begin{gather*}
\vec{\alpha} =\alpha_{1},\alpha_{1},\ldots,\alpha_{p},\qquad
\vec{\beta}+1 =\beta_{1}+1,\beta_{2}+1,\ldots,\beta_{q}+1.
\end{gather*}

Note that we have
\begin{gather*}
\frac{\rho ( x;\vec{\alpha},\vec{\beta},c )
}{\rho ( x;\vec{\alpha},\vec{\beta},c )
}=\frac{\eta(x) }{\phi ( x+1 ) },
\end{gather*}
with%
\begin{align}\label{phi-eta}
\phi(x) =x ( x+\beta_{1} ) ( x+\beta
_{2} ) \cdots ( x+\beta_{q} ) ,\qquad
\eta(x) =c ( x+\alpha_{1} ) ( x+\alpha
_{2} ) \cdots ( x+\alpha_{p} ) .
\end{align}
Hence, the weight function $\rho ( x;\vec{\alpha
},\vec{\beta},c ) $ satisf\/ies the \emph{Pearson equation} (see~\cite{Pearson} or \cite[(6.3)]{MR1761401})
\begin{gather*}
\Delta ( \phi\rho ) = ( \eta-\phi ) \rho,
\end{gather*}
where
\begin{gather*}
\Delta f(x)=f(x+1)-f(x)
\end{gather*}
is the forward dif\/ference operator. If we def\/ine $s$ by
\begin{gather*}
s=\max\big\{ \deg(\phi) -2,\deg(\phi-\eta)
-1\big\} ,
\end{gather*}
the sequence of polynomials orthogonal with respect to~$\mathcal{L}$ is
called \emph{semiclassical} of class~$s$~\cite{MR2936305,MR1270222}.

Weight functions of the form~(\ref{weight}) are also related to discrete
probability distributions (Poisson, Pascal, binomial, hypergeometric, etc.)~\cite{MR1224449}\footnote{We thank one of the editors for suggesting this reference.}. The Generalized Hypergeometric probability distributions were
studied by Adrienne W.~Kemp in her Ph.D.\ Thesis~\cite{Kemp}\footnote{Unfortunately, we haven't been able to obtain a copy of it.}. An
excellent re\-ference outlining the connections between the theory of
probability and orthogonal polynomials is~\cite{MR1761401}.

In \cite{MR3055670}, it was pointed out that since
\begin{gather}
\mu_{n}(c) =
\sum\limits_{x=0}^{\infty}
x^{n}\rho\big( x;\vec{\alpha},\vec{\beta},c\big) ,
\label{mun def}
\end{gather}
one has
\begin{gather}
\mu_{n+1}(c) =
\sum\limits_{x=0}^{\infty}
x^{n}x\rho\big( x;\vec{\alpha},\vec{\beta},c\big)
=\vartheta\mu_{n}(c) , \label{reqmu}
\end{gather}
where the dif\/ferential operator $\vartheta$ is def\/ined by \cite[(16.8.2)]{MR2723248}%
\begin{gather}
\vartheta f(c)=c\frac{df}{dc}. \label{theta}
\end{gather}
Successive applications of (\ref{reqmu}) give
\begin{gather*}
\mu_{n}=\vartheta^{n}\mu_{0}, 
\end{gather*}
and it follows that the f\/irst moment~$\mu_{0}$ determines the whole sequence
$\{ \mu_{n}\}$. If we use the operational formula~\cite{MR2798171}
\begin{gather*}
\vartheta^{k}=\sum\limits_{k=0}^{n}
\genfrac{\{}{\}}{0pt}{}{n}{k}
c^{k}\frac{d^{k}}{dc^{k}}, 
\end{gather*}
we have
\begin{gather}
\mu_{n}(c) =
\sum\limits_{k=0}^{n}
\genfrac{\{}{\}}{0pt}{}{n}{k}
c^{k}\frac{d^{k}\mu_{0}}{dc^{k}}, \label{munD}
\end{gather}
where $\genfrac{\{}{\}}{0pt}{}{n}{k}$ denote the Stirling numbers of the second kind def\/ined by~\cite[(26.8)]{MR2723248}
\begin{gather}
\genfrac{\{}{\}}{0pt}{}{n}{k}
=\frac{1}{k!}
{\displaystyle\sum\limits_{j=0}^{k}}
\binom{k}{j}(-1) ^{k-j}j^{n}. \label{stirling}
\end{gather}

From (\ref{weight}), we have
\begin{gather}
\mu_{0}(c) =
 \sum\limits_{x=0}^{\infty}
\rho\big( x;\vec{\alpha},\vec{\beta},c\big)
= {}_{p}F_{q}\left[
\begin{matrix}
\vec{\alpha}\\
\vec{\beta}
\end{matrix}
;c\right] , \label{mu0}
\end{gather}
where ${}_{p}F_{q}$ is the generalized hypergeometric function~\cite[(16.2.1)]{MR2723248}. Depending on the values of~$p$ and~$q$, we have to consider three
dif\/ferent cases:
\begin{enumerate}\itemsep=0pt
\item If $p<q+1$, $\mu_{0}(c) $ is an entire function of $c$.

\item If $p=q+1$, $\mu_{0}(c) $ is analytic inside the unit disk
$\vert c\vert <1$.

\item If $p>q+1$, the series~(\ref{mu0}) diverges for $c\neq0$, unless one or
more of the top parameters $\alpha_{i}$ is a negative integer. If we take
$\alpha_{1}=-N$, with $N\in\mathbb{N}$, then $\mu_{0}(c)$
becomes a polynomial of degree~$N$.
\end{enumerate}

Using the formula \cite[(16.3.1)]{MR2723248}
\begin{gather*}
\frac{d^{n}}{dc^{n}}\,{}_{p}F_{q}\left[
\begin{matrix}
\vec{a}\\
\vec{b}
\end{matrix}
;c\right] =\frac{( \vec{a})_{n}}{(
\vec{b})_{n}}\,{}_{p}F_{q}\left[
\begin{matrix}
\vec{a}+n\\
\vec{b}+n
\end{matrix}
;c\right] ,
\end{gather*}
in (\ref{munD}), we have
\begin{gather}
\mu_{n}(c) =
\sum\limits_{k=0}^{n}
\genfrac{\{}{\}}{0pt}{}{n}{k}
c^{k}\frac{( \vec{\alpha})_{k}}{(
\vec{\beta}+1)_{k}}\,{}_{p}F_{q}\left[
\begin{matrix}
\vec{\alpha}+k\\
\vec{\beta}+1+k
\end{matrix}
;c\right] . \label{mun gen}
\end{gather}
Although~(\ref{mun gen}) seems to give an explicit formula for the moments,
this type of sums (to our knowledge) can't be evaluated in closed form.

An alternative is to consider \emph{generalized moments}, def\/ined by~\cite{MR1688958}
\begin{gather*}
\nu_{n}=\mathcal{L} ( \varphi_{n}) ,
\end{gather*}
where $\{ \varphi_{n}\} \subset\mathbb{C}[x] $,
with $\deg(\varphi_{n}) =n$. Choosing\footnote{As suggested by
one of the anonymous referees.}
\begin{gather*}
\varphi_{n}(x) =(x-n+1)_{n},
\end{gather*}
we get%
\begin{gather*}
\nu_{n}(c) =
\sum\limits_{x=0}^{\infty}
(x-n+1)_{n}\frac{( \vec{\alpha})_{x}
}{( \vec{\beta}+1)_{x}}\frac{c^{x}}{x!}=\frac{(
\vec{\alpha})_{k}}{( \vec{\beta}+1)
_{k}}\,{}_{p}F_{q}\left[
\begin{matrix}
\vec{\alpha}+k\\
\vec{\beta}+1+k
\end{matrix}
;c\right] .
\end{gather*}
Since \cite[(26.8.10)]{MR2723248}
\begin{gather}
x^{n}=
\sum\limits_{k=1}^{n}
\genfrac{\{}{\}}{0pt}{}{n}{k}
(x-k+1)_{k}, \label{inversion1}
\end{gather}
we have
\begin{gather*}
\mu_{n}(c) =\mathcal{L}\big( x^{n}\big) =
\sum\limits_{k=1}^{n}
\genfrac{\{}{\}}{0pt}{}{n}{k}
\mathcal{L}( \varphi_{k}) =
\sum\limits_{k=1}^{n}
\genfrac{\{}{\}}{0pt}{}{n}{k}
\nu_{k}(c) ,
\end{gather*}
and we recover~(\ref{mun gen}).

In a series of papers \cite{MR2364955,MR2401156,MR2427672,MR2741218,MR2966475,MR3277948}, we studied polynomial
solutions of dif\/ferential-dif\/ference equations of the form%
\begin{gather}
P_{n+1}(x)=A_{n}(x)P_{n}^{\prime}(x)+B_{n}(x)P_{n}(x),\qquad n\geq0,
\label{eq:diff2}
\end{gather}
where $P_{0}(x)=1$, and $A_{n}(x)$, $B_{n}(x)$ are polynomials of degree at
most $2$ and $1$ respectively. In this article, we consider some extensions of~(\ref{eq:diff2}) to the multidimensional case, with $P_{n}(x)$ replaced by a~vector $\vec{P}_{n}(x)$, and $B_{n}(x)$ replaced by
a matrix~$\mathbf{B}_{n}(x)$.

In \cite{MR1451259} and~\cite{MR1627382}, it was shown that some families of
orthogonal polynomials can be represented as moments of probability measures.
In this work, we do the opposite and express the mo\-ments~$\mu_{n}(c)$ in terms of some polynomials~$\vec{P}_{n}(c)$.

The paper is organized as follows: in Section~\ref{section2}, we derive a
dif\/ferential-dif\/ference equation for the polynomials $\vec{P}_{n}(c) $ associated to the moments~$\mu_{n}(c)$.
We also f\/ind formulas for the exponential generating functions and Stieltjes
transforms of $\mu_{n}(c) $ and $\vec{P}_{n}\left(
c\right) $. In Section~\ref{section3}, we apply our results to most of the families of
discrete semiclassical orthogonal polynomials of class~1 studied in~\cite{MR3227440}, except for limiting and $c=1$ cases. Finally, in Section~\ref{section4}, we outline some conclusions and possible future directions.

\section{Main results}\label{section2}

The function $\mu_{0}(c) $ satisf\/ies the dif\/ferential equation
\cite[(16.8.3)]{MR2723248}
\begin{gather}
\big[ \vartheta ( \vartheta+\beta_{1} ) \cdots (
\vartheta+\beta_{q} ) -c ( \vartheta+\alpha_{1} )
\cdots ( \vartheta+\alpha_{p} ) \big] \mu_{0}=0, \label{ODE}
\end{gather}
where $\vartheta$ was def\/ined in~(\ref{theta}). We can rewrite the ODE~(\ref{ODE}) as
\begin{gather}
\vartheta^{q+1}\mu_{0} =
\sum\limits_{k=0}^{q}
\sigma_{k}(c) \vartheta^{k}\mu_{0},\qquad q>p-1,\nonumber\\
(1-c) \vartheta^{q+1}\mu_{0} =
\sum\limits_{k=0}^{q}
\sigma_{k}(c) \vartheta^{k}\mu_{0},\qquad q=p-1,\label{diffmu0}\\
c\vartheta^{p}\mu_{0} =
\sum\limits_{k=0}^{p-1}
\sigma_{k}(c) \vartheta^{k}\mu_{0},\qquad q<p-1,\nonumber
\end{gather}
where the coef\/f\/icients $\sigma_{k}(c) $ are linear functions of~$c$.

Introducing the quantities
\begin{gather*}
\vec{\mu}(c) =
\begin{bmatrix}
\mu_{0}(c) \\
\mu_{1}(c) \\
\vdots\\
\mu_{\xi}(c)
\end{bmatrix}
,\qquad\xi=\max \{ p-1,q \} ,
\end{gather*}
and
\begin{gather*}
(\lambda,\tau) =
\begin{cases}
(1,0) , & q+1>p, \\
(1,-1) , & q+1=p, \\
(0,1) , & q+1<p,
\end{cases}
\end{gather*}
we can rewrite (\ref{diffmu0}) as%
\begin{gather*}
( \lambda+c\tau) \mu_{\xi+1}=
\sum\limits_{k=0}^{\xi}
\sigma_{k}(c) \mu_{k}.
\end{gather*}
If we def\/ine the $(\xi+1) \times(\xi+1) $ matrix
$\mathbf{M}(c) $ by
\begin{gather*}
\mathbf{M}_{i,j}=
\begin{cases}
\sigma_{j}, & i=\xi+1,\quad 0\leq j\leq\xi,\\
\lambda+c\tau, & j=i+1,\quad 0\leq i\leq\xi-1,\\
0, & \text{otherwise},
\end{cases}
\end{gather*}
we get%
\begin{gather}
\mathbf{M}\,\vec{\mu}= ( \lambda+c\tau ) \vartheta
 \vec{\mu}.\label{Mmu}
\end{gather}
We can now state our main result.

\begin{proposition}
\label{main}Let the $(\xi+1)$-vector polynomials
$\vec{P}_{n}(c) $ be defined by
\begin{gather}
\vec{P}_{0}(c) =
\begin{bmatrix}
1\\
0\\
\vdots\\
0
\end{bmatrix}
\label{P0}
\end{gather}
and
\begin{gather}
\vec{P}_{n+1}=c ( \lambda+c\tau ) \frac
{d \vec{P}_{n}}{dc}+\big( \mathbf{M}^{T}-n\tau c\mathbf{I}
\big) \vec{P}_{n},\qquad n=0,1,\ldots,\label{reqPn}
\end{gather}
where $\mathbf{I}$ is the $(\xi+1) \times(\xi+1) $
identity matrix. Then,
\begin{gather}
\mu_{n}(c) =(\lambda+c\tau) ^{-n}
 \vec{P}_{n}(c) \cdot\vec{\mu} (c) ,\qquad n=0,1,\ldots. \label{Pmu}
\end{gather}
\end{proposition}

\begin{proof}
Using (\ref{Mmu}) and (\ref{reqPn}), we have
\begin{gather*}
\vec{P}_{n+1}\cdot\vec{\mu} = ( \lambda
+c\tau ) \vartheta \vec{P}_{n}\cdot\vec{\mu
}+\big( \mathbf{M}^{T}-n\tau c\mathbf{I}\big) \vec{P}
_{n}\cdot\vec{\mu}\\
\hphantom{\vec{P}_{n+1}\cdot\vec{\mu}}{}
 = ( \lambda+c\tau ) \vartheta \vec{P}_{n}
\cdot\vec{\mu}+\vec{P}_{n}\cdot ( \mathbf{M}-n\tau
c\mathbf{I} ) \vec{\mu}\\
\hphantom{\vec{P}_{n+1}\cdot\vec{\mu}}{}
 = ( \lambda+c\tau ) \vartheta \vec{P}_{n}
\cdot\vec{\mu}+\vec{P}_{n}\cdot ( \lambda
+c\tau ) \vartheta \vec{\mu} -n\tau c \vec{P}_{n}\cdot\vec{\mu}.
\end{gather*}
Multiplying by $( \lambda+c\tau) ^{-n-1}$ we get
\begin{gather*}
(\lambda+c\tau) ^{-n-1} \vec{P}_{n+1}\cdot
\vec{\mu} =( \lambda+c\tau) ^{-n}\vartheta
\vec{P}_{n}\cdot\vec{\mu}+\vec{P}_{n}
\cdot(\lambda+c\tau) ^{-n}\vartheta \vec{\mu} -n\tau c(\lambda+c\tau) ^{-n-1} \vec{P}_{n}
\cdot\vec{\mu}\\
\hphantom{(\lambda+c\tau) ^{-n-1} \vec{P}_{n+1}\cdot \vec{\mu}}{}
 =\vartheta\big[ (\lambda+c\tau) ^{-n} \vec{P}_{n}\cdot\vec{\mu}\big] .
\end{gather*}
Thus, the sequence
\begin{gather*}
r_{n}=(\lambda+c\tau) ^{-n}\vec{P}_{n}\cdot \vec{\mu}
\end{gather*}
satisf\/ies the recurrence $r_{n+1}=\vartheta r_{n}$ with initial condition
\begin{gather*}
r_{0}=\vec{P}_{0}\cdot\vec{\mu}=\mu_{0}.
\end{gather*}
From (\ref{reqmu}), we conclude that $r_{n}=\mu_{n}$.
\end{proof}

\subsection{Generating functions}\label{section2.1}

Let's consider the exponential generating function for the moments, def\/ined by
the formal power series
\begin{gather*}
G_{\mu}(c,w)=
\sum\limits_{n=0}^{\infty}
\mu_{n}(c) \frac{w^{n}}{n!}.
\end{gather*}
Given that
\begin{gather*}
\mu_{n+1}=c\mu_{n}^{\prime},
\end{gather*}
we have \cite{MR2172781}
\begin{gather*}
\frac{\partial G_{\mu}}{\partial w}=c\frac{\partial G_{\mu}}{\partial c},
\end{gather*}
with general solution
\begin{gather*}
G_{\mu}(c,w)=H\big( ce^{w}\big)
\end{gather*}
for some function $H(w)$. But since
\begin{gather*}
G_{\mu}(c,0)=\mu_{0}(c) ,
\end{gather*}
we conclude that
\begin{gather}
G_{\mu}(c,w)=\mu_{0}\big( ce^{w}\big) . \label{Gmu}
\end{gather}

Using (\ref{Pmu}), we see that the exponential generating function for the
polynomials $\vec{P}_{n}(c)$
\begin{gather*}
\vec{G}_{P}(c,w)=
\sum\limits_{n=0}^{\infty}
\vec{P}_{n}(c) \frac{w^{n}}{n!}
\end{gather*}
satisf\/ies
\begin{gather*}
G_{\mu}( c,( \lambda+\tau c) w) =\vec{G}_{P}(c,w)\cdot\vec{\mu}(c) ,
\end{gather*}
or, using (\ref{Gmu}),
\begin{gather}
\vec{G}_{P}(c,w)\cdot\vec{\mu}(c) =\mu
_{0}\big( ce^{( \lambda+\tau c) w}\big) . \label{GPn}
\end{gather}

\subsection{Stieltjes transform}\label{section2.2}

A dif\/ferent type of generating function for the moments that is very important
in the theory of orthogonal polynomials is the Stieltjes transform (or~$Z$
transform), that can be def\/ined by the formal Laurent series
\begin{gather*}
S_{\mu}(c,z)=
\sum\limits_{n=0}^{\infty}
\frac{\mu_{n}(c) }{z^{n+1}}.
\end{gather*}
We have
\begin{gather*}
S_{\mu}(c,z) =
\sum\limits_{n=0}^{\infty}
\frac{\mu_{n}(c) }{z^{n+1}}=
\sum\limits_{n=-1}^{\infty}
\frac{\mu_{n+1}(c) }{z^{n+2}}=\frac{1}{z}
\sum\limits_{n=-1}^{\infty}
\frac{\mu_{n+1}(c) }{z^{n+1}}
 =\frac{1}{z}\left[ \mu_{0}(c) +
\sum\limits_{n=0}^{\infty}
\frac{\mu_{n+1}(c) }{z^{n+1}}\right] .
\end{gather*}
Hence,
\begin{gather}
\sum\limits_{n=0}^{\infty}
\frac{\mu_{n+1}(c) }{z^{n+1}}=zS_{\mu}(c,z)-\mu_{0}(c) \label{shift}
\end{gather}
and using (\ref{reqmu}), we get
\begin{gather}
c\frac{\partial S_{\mu}}{\partial c}=zS_{\mu}-\mu_{0}(c).\label{PDES}
\end{gather}
From (\ref{mun def}), we have
\begin{gather*}
S_{\mu}(0,z)=\frac{1}{z}
\end{gather*}
and solving (\ref{PDES}) we obtain
\begin{gather*}
S_{\mu}(c,z)=-c^{z}
\int_{0}^{c}
\frac{\mu_{0}(x) }{x^{z+1}}dx=-
\int_{0}^{1}
\frac{\mu_{0}(ct) }{t^{z+1}}dt.
\end{gather*}
Using the recurrence relation for the Gamma function, we can write%
\begin{gather*}
-
\int_{0}^{1}
\frac{\mu_{0}(ct) }{t^{z+1}}dt=\frac{1}{z}\frac{\Gamma(
1-z) }{\Gamma(1) \Gamma(-z)}
\int_{0}^{1}
\frac{\mu_{0}(ct) }{t^{z+1}}dt,
\end{gather*}
and therefore
\begin{gather}
S_{\mu}(c,z)=\frac{1}{z}\,{}_{p+1}F_{q+1}\left[
\begin{matrix}
-z,\ \vec{\alpha}\\
1-z,\ \vec{\beta}+1
\end{matrix}
;c\right] ,\label{stiltjes}
\end{gather}
where we have used the integral representation \cite[(16.5.2)]{MR2723248}
\begin{gather*}
{}_{p+1}F_{q+1}\left[
\begin{matrix}
\alpha_{0},\ \vec{\alpha}\\
\beta_{0},\ \vec{\beta}
\end{matrix}
;c\right]
 =\frac{\Gamma ( \beta_{0}) }{\Gamma( \alpha_{0})
\Gamma( \beta_{0}-\alpha_{0}) }
\int_{0}^{1}
t^{\alpha_{0}-1}( 1-t) ^{\beta_{0}-\alpha_{0}-1}\,{}_{p}
F_{q}\left[
\begin{matrix}
\vec{\alpha}\\
\vec{\beta}
\end{matrix}
;ct\right] dt.
\end{gather*}
We derived (\ref{stiltjes}) in \cite{MH} using a dif\/ferent approach.

If we def\/ine the Stieltjes transform of $\vec{P}_{n}(c) $ by
\begin{gather*}
\vec{S}_{P}(c,z) =
\sum\limits_{n=0}^{\infty}
\frac{\vec{P}_{n}(c) }{z^{n+1}},
\end{gather*}
then it follows from (\ref{Pmu}) that
\begin{gather*}
\frac{1}{\lambda+\tau c}S_{\mu}\left( c,\frac{z}{\lambda+\tau c}\right)
=\vec{S}_{P}(c,z) \cdot\vec{\mu}(c) ,
\end{gather*}
or
\begin{gather}
\vec{S}_{P}(c,z) \cdot\vec{\mu} (
c ) =\frac{1}{z}\,{}_{p+1}F_{q+1}\left[
\begin{matrix}
-\frac{z}{\lambda+\tau c},\ \vec{\alpha}\\
1-\frac{z}{\lambda+\tau c},\ \vec{\beta}+1
\end{matrix}
;c\right] . \label{laplaceG}
\end{gather}

\section{Examples}\label{section3}

\strut In \cite{MR3227440} we studied all families of semiclassical
polynomials of class $s\leq1$ orthogonal with respect to~(\ref{L}). When
$s=0$, we have three canonical cases (the \emph{discrete classical
polynomials}):
$$
\begin{tabular}
[c]{|c|c|c|}\hline
$\deg(\eta) $ & $\deg(\phi) $ & \\\hline
$0$ & $1$ & Charlier\\\hline
$1$ & $1$ & Meixner\\\hline
$2$ & $2$ & Hahn\\\hline
\end{tabular}
$$
where $\phi(x) $ and $\eta(x) $ were def\/ined in~(\ref{phi-eta}). These polynomials are associated with the Poisson, Pascal,
and hypergeometric probability distributions, and results about their moments
have appeared in many places before (see~\cite{MR1224449}, for instance).

In \cite{MR3286553}, the authors used (\ref{inversion1}) and inversion
formulas of the form
\begin{gather*}
(x-n+1)_{n}=\sum\limits_{k=0}^{n}
c_{n,k}Q_{k}(x) ,
\end{gather*}
to derive expressions for the moments and exponential generating functions of
the discrete classical polynomials. Stieltjes' transforms were not considered.

When $s=1$, we obtained f\/ive cases:%
$$
\begin{tabular}
[c]{|c|c|c|}\hline
$\deg(\eta) $ & $\deg(\phi) $ & \\\hline
$0$ & $2$ & generalized Charlier\\\hline
$1$ & $2$ & generalized Meixner\\\hline
$2$ & $1$ & generalized Krawtchouk\\\hline
$2$ & $2$ & generalized Hahn of type I\\\hline
$3$ & $3$ & generalized Hahn of type II\\\hline
\end{tabular}
$$
The moments of these polynomials have not (to our knowledge) been studied before.

\subsection{Charlier}\label{section3.1}

The Charlier polynomials \cite[(18.20.8)]{MR2723248}
\begin{gather*}
C_{n}(x;c) ={}_{2}F_{0}\left(
\begin{matrix}
-n,-x\\
-
\end{matrix}
;-\frac{1}{c}\right)
\end{gather*}
are orthogonal with respect to the weight function \cite{MR2300274}
\begin{gather*}
\rho_{C}(x) =\frac{c^{x}}{x!},\qquad c>0.
\end{gather*}
In this case, we have
\begin{gather*}
\mu_{0}(c) ={}_{0}F_{0}\left[
\begin{matrix}
-\\
-
\end{matrix}
;c\right] =e^{c},
\end{gather*}
and hence
\begin{gather*}
\xi=0,\qquad(\lambda,\tau) =(1,0) .
\end{gather*}
From (\ref{ODE}) we see that $\mu_{0}(c) $ satisf\/ies the ODE
\begin{gather*}
(\vartheta-c) \mu_{0}=0,
\end{gather*}
which implies $\mu_{1}=c\mu_{0}$. Thus, $\sigma_{0}(c) =c$, and
Proposition~\ref{main} gives
\begin{gather*}
\mu_{n}(c) =P_{n}(c) \mu_{0}(c) ,
\end{gather*}
with $P_{n}(c) $ def\/ined by $P_{0}(c) =1$ and
\begin{gather}
P_{n+1}=c\frac{dP_{n}}{dc}+cP_{n}. \label{Bell}
\end{gather}
The polynomials satisfying~(\ref{Bell}) are known as Bell (or Touchard, or
exponential) polynomials. It is well known that they have the explicit
representation~\cite{MR2583007}
\begin{gather}
P_{n}(c) =
\sum\limits_{k=0}^{n}
\genfrac{\{}{\}}{0pt}{}{n}{k}
c^{k}, \label{bell}
\end{gather}
and therefore (see also \cite[equation~(44)]{MR3286553})
\begin{gather*}
\mu_{n}(c) =e^{c}
\sum\limits_{k=0}^{n}
\genfrac{\{}{\}}{0pt}{}{n}{k}
c^{k},
\end{gather*}
in agreement with (\ref{mun gen}).

Using (\ref{Gmu}) and (\ref{GPn}), we see that the generating functions of
$\mu_{n}(c) $ and $P_{n}(c) $ are given by (see
also \cite[equation~(50)]{MR3286553})
\begin{gather*}
G_{\mu}(c,w)=\mu_{0}\big( ce^{w}\big) =e^{ce^{w}}, 
\end{gather*}
and
\begin{gather*}
G_{P}(c,w)=\frac{e^{ce^{w}}}{e^{c}}=e^{c( e^{w}-1) }.
\end{gather*}

\subsection{Meixner}\label{section3.2}

The Meixner polynomials \cite[(18.20.7)]{MR2723248}
\begin{gather*}
M_{n}(x;\alpha,c) ={}_{2}F_{1}\left(
\begin{matrix}
-n,-x\\
\alpha
\end{matrix}
;1-\frac{1}{c}\right)
\end{gather*}
are orthogonal with respect to the weight function \cite[(18.19)]{MR2723248}
\begin{gather}
\rho_{M}(x;\alpha,c) =(\alpha)_{x}\frac{c^{x}
}{x!},\qquad 0<c<1,\qquad\alpha>0. \label{rhomeixner}
\end{gather}
In this case, we have%
\begin{gather}
\mu_{0}^{M}(c;\alpha) ={}_{1}F_{0}\left[
\begin{matrix}
\alpha\\
-
\end{matrix}
;c\right] =(1-c) ^{-\alpha}, \label{mu0meixner}
\end{gather}
and hence
\begin{gather*}
\xi=0,\qquad(\lambda,\tau) =(1,-1) .
\end{gather*}
From (\ref{ODE}) we see that $\mu_{0}^{M}(c;\alpha) $ satisf\/ies
the ODE%
\begin{gather*}
[ \vartheta-c(\vartheta+\alpha) ] \mu_{0}^{M}=0,
\end{gather*}
which implies%
\begin{gather*}
(1-c) \mu_{1}^{M}=\alpha c\mu_{0}^{M}.
\end{gather*}
Thus, $\sigma_{0}(c) =\alpha c$, and Proposition \ref{main} gives
\begin{gather}
\mu_{n}^{M}(c;\alpha) =(1-c) ^{-n} P_{n}
^{M}(c) \mu_{0}^{M}(c;\alpha) , \label{mu meixner}
\end{gather}
with $P_{n}^{M}(c) $ def\/ined by $P_{0}^{M}(c) =1$
and (see also \cite[equation~(5.68)]{Njionou})
\begin{gather}
P_{n+1}^{M}=c(1-c) \frac{dP_{n}^{M}}{dc}+ ( \alpha +n) cP_{n}^{M}. \label{DDMeixner}
\end{gather}

From (\ref{mun gen}), we have (see also \cite[equation~(43)]{MR3286553})
\begin{gather*}
\mu_{n}^{M}(c;\alpha) =
\sum\limits_{k=0}^{n}
\genfrac{\{}{\}}{0pt}{}{n}{k}
c^{k}(\alpha)_{k} (1-c) ^{-(\alpha+k) },
\end{gather*}
and therefore
\begin{gather}
P_{n}^{M}(c) =
\sum\limits_{k=0}^{n}
\genfrac{\{}{\}}{0pt}{}{n}{k}
(\alpha)_{k} c^{k}(1-c) ^{n-k}. \label{meixner}
\end{gather}

Using (\ref{Gmu}) and (\ref{GPn}), we see that the generating functions of
$\mu_{n}^{M}(c;\alpha) $ and $P_{n}^{M}(c) $ are
given by (see also \cite[equation~(49)]{MR3286553})
\begin{gather*}
G_{\mu}(c,w)=\mu_{0}\big( ce^{w}\big) =\big( 1-ce^{w}\big)
^{-\alpha}, 
\end{gather*}
and
\begin{gather*}
G_{P}(c,w)=\frac{\big[ 1-ce^{(1-c) w}\big] ^{-\alpha}
}{(1-c) ^{-\alpha}}=\left[ \frac{1-ce^{(1-c) w}}{1-c}\right] ^{-\alpha}.
\end{gather*}

\subsubsection{Krawtchouk polynomials}

\strut The Krawtchouk polynomials \cite[(18.20.6)]{MR2723248}
\begin{gather*}
K_{n} ( x;\mathfrak{p,}N ) ={}_{2}F_{1}\left(
\begin{matrix}
-n,-x\\
-N
\end{matrix}
;\frac{1}{\mathfrak{p}}\right) ,
\end{gather*}
are orthogonal with respect to the weight function \cite{MR2390273}
\begin{gather*}
\rho_{K} ( x;\mathfrak{p,}N ) =\binom{N}{x}\mathfrak{p}^{x} (
1-\mathfrak{p} ) ^{N-x},\qquad N\in\mathbb{N},\qquad 0<\mathfrak{p}<1.
\end{gather*}
It is related to the weight function of the Meixner polynomials~(\ref{rhomeixner}) by
\begin{gather}
\rho_{K} ( x;\mathfrak{p,}N ) = ( 1-\mathfrak{p} )
^{N}\frac{( -N)_{x}}{x!}\left( \frac{\mathfrak{p}}
{\mathfrak{p}-1}\right) ^{x}=( 1-\mathfrak{p}) ^{N}\rho
_{M}\left( x;-N,\frac{\mathfrak{p}}{\mathfrak{p}-1}\right) .\label{rhokrawt}
\end{gather}

Using (\ref{rhokrawt}) in (\ref{mu0meixner}), we obtain
\begin{gather*}
\mu_{0}^{K} ( \mathfrak{p} ) = ( 1-\mathfrak{p} )
^{N}\mu_{0}^{M}\left( -N,\frac{\mathfrak{p}}{\mathfrak{p}-1}\right) =1.
\end{gather*}
and using (\ref{mu meixner}) and (\ref{meixner}), we get $\mu_{n}^{K} (
\mathfrak{p} ) = P_{n}^{K} ( \mathfrak{p} ) $, with (see
also \cite[equation~(42)]{MR3286553})
\begin{gather*}
P_{n}^{K} ( \mathfrak{p} ) =
 \sum\limits_{k=0}^{n}
\genfrac{\{}{\}}{0pt}{}{n}{k}
\binom{N}{k}k! \mathfrak{p}^{k}. 
\end{gather*}
From (\ref{DDMeixner}) and (\ref{rhokrawt}), we have
\begin{gather*}
P_{n+1}^{K}=\mathfrak{p} ( 1-\mathfrak{p} ) \frac{dP_{n}^{K}
}{d\mathfrak{p}}+N\mathfrak{p}P_{n}^{K}.
\end{gather*}

To obtain the generating function of $P_{n}^{K}( \mathfrak{p})$,
we can use \cite[(26.8.12)]{MR2723248}
\begin{gather}
\sum\limits_{n=0}^{\infty}
\genfrac{\{}{\}}{0pt}{}{n}{k}
\frac{w^{n}}{n!}=\frac{( e^{w}-1) ^{k}}{k!}, \label{gen Stirling}
\end{gather}
and we get (see also \cite[equation~(48)]{MR3286553})
\begin{gather*}
G_{P}(\mathfrak{p},w)=\big[ 1+\big( e^{w}-1\big) \mathfrak{p}\big]^{N}.
\end{gather*}

\subsection{Generalized Charlier}\label{section3.3}

These polynomials are orthogonal with respect to the weight function~\cite{MR2996960}
\begin{gather*}
\rho(x;\beta,c) =\frac{1}{(\beta+1)_{x}}
\frac{c^{x}}{x!},\qquad c>0,\qquad \beta>-1.
\end{gather*}
In this case, we have
\begin{gather*}
\mu_{0}(c) ={}_{0}F_{1}\left[
\begin{matrix}
-\\
\beta+1
\end{matrix}
;c\right] =c^{-\frac{\beta}{2}}\Gamma(\beta+1) I_{\beta}\big(
2\sqrt{c}\big) ,
\end{gather*}
where $I_{\beta}(z) $ is the modif\/ied Bessel function of the
f\/irst kind \cite[(10.39.9)]{MR2723248}, and hence
\begin{gather*}
\xi=1,\qquad(\lambda,\tau) =(1,0) .
\end{gather*}
From (\ref{ODE}) we see that $\mu_{0}(c) $ satisf\/ies the ODE
\begin{gather*}
\vartheta ( \vartheta+\beta ) \mu_{0}-c\mu_{0}=0,
\end{gather*}
which implies
\begin{gather*}
\mu_{2}=c\mu_{0}-\beta\mu_{1},
\end{gather*}
and therefore
\begin{gather*}
\sigma_{0}(c) =c,\qquad\sigma_{1}(c) =-\beta.
\end{gather*}
Proposition \ref{main} gives
\begin{gather*}
\mu_{n}(c) =\vec{P}_{n}(c) \cdot\vec{\mu}(c) ,
\end{gather*}
with $\vec{P}_{n}(c) $ def\/ined by~(\ref{P0}) and
\begin{gather}
\vec{P}_{n+1}=c\frac{d \vec{P}_{n}}{dc}+
\begin{bmatrix}
0 & c\\
1 & -\beta
\end{bmatrix}
 \vec{P}_{n}.\label{dd gen charlier}
\end{gather}

To obtain the Stieltjes transform of $\vec{P}_{n}(c)$, let's write
\begin{gather*}
\vec{P}_{n}(c) =\left[
\begin{matrix}
Q_{n}(c) \\
R_{n}(c)
\end{matrix}
\right] ,\qquad\vec{S}_{P}(c,z) =\left[
\begin{matrix}
U(c,z)\\
V(c,z)
\end{matrix}
\right] .
\end{gather*}
From (\ref{dd gen charlier}), we have
\begin{gather*}
Q_{n+1} =cQ_{n}^{\prime}+cR_{n},\qquad
R_{n+1} =cR_{n}^{\prime}-\beta R_{n}+Q_{n},
\end{gather*}
with $Q_{0}=1$, $R_{0}=0$. Using~(\ref{shift}), we get
\begin{gather}\label{UV laplace}
zU-1 =c\frac{\partial U}{\partial c}+cV, \qquad
zV =c\frac{\partial V}{\partial c}-\beta V+U.
\end{gather}
The solution of the system (\ref{UV laplace}) is given by
\begin{gather*}
U(c,z) =\frac{1}{z}\,{}_{1}F_{2}\left[
\begin{matrix}
1\\
1-z,\ -\beta-z
\end{matrix}
;c\right] ,\qquad
V(c,z) =\frac{1}{z ( z+\beta ) }\,{}_{1}
F_{2}\left[
\begin{matrix}
1\\
1-z,\ 1-\beta-z
\end{matrix};c\right] .
\end{gather*}

In this case
\begin{gather*}
\mu_{0}(c) ={}_{0}F_{1}\left[
\begin{matrix}
-\\
\beta+1
\end{matrix}
;c\right] ,\qquad
\mu_{1}(c) =\frac{c}{\beta+1}\,{}_{0}F_{1}\left[
\begin{matrix}
-\\
\beta+2
\end{matrix}
;c\right] ,
\end{gather*}
and since
\begin{gather*}
 \frac{1}{z}\,{}_{1}F_{2}\left[
\begin{matrix}
1\\
1-z,\ -\beta-z
\end{matrix}
;c\right] \,{}_{0}F_{1}\left[
\begin{matrix}
-\\
\beta+1
\end{matrix}
;c\right] \\
 \qquad{} +\frac{1}{z(z+\beta) }\,{}_{1}F_{2}\left[
\begin{matrix}
1\\
1-z,\ 1-\beta-z
\end{matrix}
;c\right] \frac{c}{\beta+1}\,{}_{0}F_{1}\left[
\begin{matrix}
-\\
\beta+2
\end{matrix}
;c\right]  =\frac{1}{z}\,{}_{1}F_{2}\left[
\begin{matrix}
-z\\
1-z,\ \beta+1
\end{matrix}
;c\right] ,
\end{gather*}
we see that (\ref{laplaceG}) is satisf\/ied.

\subsection{Generalized Meixner}\label{section3.4}

\strut These polynomials are orthogonal with respect to the weight function~\cite{MR2749070}
\begin{gather*}
\rho( x;\alpha,\beta,c) =\frac{(\alpha)_{x}
}{(\beta+1)_{x}}\frac{c^{x}}{x!},\qquad c,\alpha>0,\qquad
\beta>-1.
\end{gather*}
In this case, we have
\begin{gather*}
\mu_{0}(c) ={}_{1}F_{1}\left[
\begin{matrix}
\alpha\\
\beta+1
\end{matrix}
;c\right] =M ( \alpha,\beta+1,c ) ,
\end{gather*}
where $M(a,b,c) $ is Kummer's (conf\/luent hypergeometric)
function \cite[(13.2.2)]{MR2723248}, and hence
\begin{gather*}
\xi=1,\qquad(\lambda,\tau) =(1,0) .
\end{gather*}
From (\ref{ODE}) we see that $\mu_{0}(c) $ satisf\/ies the ODE
\begin{gather*}
\vartheta ( \vartheta+\beta ) \mu_{0}-c ( \vartheta +\alpha ) \mu_{0}=0,
\end{gather*}
which implies
\begin{gather*}
\mu_{2}=c\alpha\mu_{0}+(c-\beta) \mu_{1},
\end{gather*}
and therefore
\begin{gather*}
\sigma_{0}(c) =\alpha c,\qquad\sigma_{1}(c) =c-\beta.
\end{gather*}
Proposition \ref{main} gives
\begin{gather*}
\mu_{n}(c) =\vec{P}_{n}(c) \cdot\vec{\mu}(c) ,
\end{gather*}
with $\vec{P}_{n}(c) $ def\/ined by~(\ref{P0}) and
\begin{gather*}
\vec{P}_{n+1}=c\frac{d \vec{P}_{n}}{dc}+ \begin{bmatrix}
0 & \alpha c\\
1 & c-\beta
\end{bmatrix}
\vec{P}_{n}.
\end{gather*}

To obtain the Stieltjes transform of $\vec{P}_{n}(c)$, let's write
\begin{gather*}
\vec{P}_{n}(c) =\left[
\begin{matrix}
Q_{n}(c) \\
R_{n}(c)
\end{matrix}
\right] ,\qquad\vec{S}_{P}(c,z) =\left[
\begin{matrix}
U(c,z)\\
V(c,z)
\end{matrix}
\right] .
\end{gather*}
From (\ref{dd gen charlier}), we have
\begin{gather*}
Q_{n+1} =cQ_{n}^{\prime}+\alpha cR_{n},\qquad
R_{n+1} =cR_{n}^{\prime}+(c-\beta) R_{n}+Q_{n},
\end{gather*}
with $Q_{0}=1$, $R_{0}=0$. Using (\ref{shift}), we get
\begin{gather}\label{UV laplace 1}
zU-1 =c\frac{\partial U}{\partial c}+\alpha cV,\qquad
zV =c\frac{\partial V}{\partial c}+(c-\beta) V+U.
\end{gather}

If we represent the functions $U$, $V$ as
\begin{gather*}
U(c,z)=
\sum\limits_{n=0}^{\infty}
u_{n}(z) c^{n},\qquad V(c,z)=
\sum\limits_{n=0}^{\infty}
v_{n}(z) c^{n},
\end{gather*}
then (\ref{UV laplace 1}) gives
\begin{gather}\label{req un vn}
zu_{n}-\delta_{n,0} =nu_{n}+\alpha v_{n-1},\qquad
zv_{n} =nv_{n}+v_{n-1}-\beta v_{n}+u_{n}.
\end{gather}
We have (assuming that $u_{n}=v_{n}=0$ for $n<0)$
\begin{gather*}
u_{0}=\frac{1}{z},\qquad v_{0}=\frac{1}{z(z+\beta) },
\end{gather*}
and%
\begin{gather}
u_{n}=(z+\beta-n) v_{n}-v_{n-1}. \label{un}
\end{gather}
Using (\ref{un}) in (\ref{req un vn}), we obtain
\begin{gather*}
v_{n}=\frac{z+\alpha-n}{( z-n) (z+\beta-n) }v_{n-1},\qquad n=1,2,\ldots,
\end{gather*}
and therefore
\begin{gather*}
v_{n}=\frac{(-1) ^{n}}{z(z+\beta) }\frac{(
1-\alpha-z)_{n}}{( 1-z)_{n}( 1-\beta-z)
_{n}},\qquad n=0,1,\ldots.
\end{gather*}
From (\ref{un}), it follows that
\begin{gather*}
u_{n}=\alpha\frac{(-1) ^{n}}{z(z+\alpha) }
\frac{(-\alpha-z)_{n}}{(1-z)_{n}(-\beta-z)_{n}},\qquad n=1,2,\ldots.
\end{gather*}
Thus,
\begin{gather*}
U(c,z) =\frac{1}{z}+\frac{\alpha}{z(z+\alpha)
}
\sum\limits_{n=1}^{\infty}
\frac{(-\alpha-z)_{n}}{(1-z)_{n}(
-\beta-z)_{n}}(-c) ^{n}\\
\hphantom{U(c,z)}{} =\frac{1}{z}-\frac{\alpha}{z(z+\alpha) }+\frac{\alpha
}{z(z+\alpha) }
\sum\limits_{n=0}^{\infty}
\frac{(-\alpha-z)_{n}}{(1-z)_{n}(
-\beta-z)_{n}}(-c) ^{n}\\
\hphantom{U(c,z)}{}
 =\frac{1}{z+\alpha}+\frac{\alpha}{z(z+\alpha) }\,{}_{2}F_{2}\left[
\begin{matrix}
1,-\alpha-z\\
1-z,\ -\beta-z
\end{matrix}
;-c\right] ,
\end{gather*}
and
\begin{gather*}
V(c,z) =\frac{1}{z(z+\beta) }\,{}_{2}F_{2}\left[
\begin{matrix}
1,\ 1-\alpha-z\\
1-z,\ 1-\beta-z
\end{matrix}
;-c\right] .
\end{gather*}

In this case
\begin{gather*}
\mu_{0}(c) ={}_{1}F_{1}\left[
\begin{matrix}
\alpha\\
\beta+1
\end{matrix}
;c\right] ,\qquad
\mu_{1}(c) =\frac{\alpha c}{\beta+1}\,{}_{1}F_{1}\left[
\begin{matrix}
\alpha+1\\
\beta+2
\end{matrix}
;c\right] ,
\end{gather*}
and since
\begin{gather*}
 \left( \frac{1}{z+\alpha}+\frac{\alpha}{z(z+\alpha) }
\,{}_{2}F_{2}\left[
\begin{matrix}
1,-\alpha-z\\
1-z,\ -\beta-z
\end{matrix}
;-c\right] \right) \,{}_{1}F_{1}\left[
\begin{matrix}
\alpha\\
\beta+1
\end{matrix}
;c\right] \\
\! \qquad {} +\frac{1}{z(z+\beta) }\,{}_{2}F_{2}\left[
\begin{matrix}
1,\ 1-\alpha-z\\
1-z,\ 1-\beta-z
\end{matrix}
;-c\right] \frac{\alpha c}{\beta+1}\,{}_{1}F_{1}\left[
\begin{matrix}
\alpha+1\\
\beta+2
\end{matrix}
;c\right]   =\frac{1}{z}\,{}_{2}F_{2}\left[
\begin{matrix}
-z,\ \alpha\\
1-z,\ \beta+1
\end{matrix}
;c\right] ,
\end{gather*}
we see that (\ref{laplaceG}) is satisf\/ied.

\subsection{Generalized Krawtchouk}\label{section3.5}

\strut These polynomials are orthogonal with respect to the weight function~\cite{MR3227440}
\begin{gather*}
\rho(x;\alpha,N,c) =(\alpha)_{x} (
-N )_{x}\frac{c^{x}}{x!},\qquad c<0,\qquad\alpha>0,\qquad N\in
\mathbb{N}.
\end{gather*}
In this case, we have
\begin{gather*}
\mu_{0}(c) ={}_{2}F_{0}\left[
\begin{matrix}
\alpha,\ -N\\
-
\end{matrix}
;c\right] =C_{N}\big({-}\alpha;-c^{-1}\big) ,
\end{gather*}
where $C_{N}(x;\mu) $ is the Charlier polynomial
\cite{MR2300274}, and hence%
\begin{gather*}
\xi=1,\qquad(\lambda,\tau) =(0,1) .
\end{gather*}
From (\ref{ODE}) we see that $\mu_{0}(c) $ satisf\/ies the ODE
\begin{gather*}
\vartheta\mu_{0}-c(\vartheta+\alpha) (\vartheta-N)
\mu_{0}=0,
\end{gather*}
which implies
\begin{gather*}
c\mu_{2}=\alpha Nc\mu_{0}+(Nc-\alpha c+1) \mu_{1},
\end{gather*}
and therefore%
\begin{gather*}
\sigma_{0}(c) =\alpha Nc,\qquad\sigma_{1}(c)
=(N-\alpha) c+1.
\end{gather*}
Proposition \ref{main} gives
\begin{gather*}
\mu_{n}(c) =c^{-n}\vec{P}_{n}(c)
\cdot\vec{\mu}(c) ,
\end{gather*}
with $\vec{P}_{n}(c) $ def\/ined by (\ref{P0}) and
\begin{gather*}
\vec{P}_{n+1}=c^{2}\frac{d \vec{P}_{n}}{dc}+
\begin{bmatrix}
-nc & \alpha Nc\\
c & \left( N-\alpha-n\right) c+1
\end{bmatrix}
\vec{P}_{n}(c) .
\end{gather*}

\subsection{Generalized Hahn polynomials of type I}\label{section3.6}

These polynomials are orthogonal with respect to the weight function~\cite{MR3227440}
\begin{gather}
\rho( x;\alpha_{1},\alpha_{2},\beta;c) =\frac{( \alpha
_{1})_{x}( \alpha_{2})_{x}}{(\beta+1)
_{x}}\frac{c^{x}}{x!}, \label{rho gen hahn}
\end{gather}
where $0<c<1$, $\alpha_{1},\alpha_{2}>0$, $\beta>-1$. In this case, we have%
\begin{gather*}
\mu_{0}(c) = {}_{2}F_{1}\left[
\begin{matrix}
\alpha_{1},\ \alpha_{2}\\
\beta+1
\end{matrix}
;c\right] ,
\end{gather*}
and hence%
\begin{gather*}
\xi=1,\qquad(\lambda,\tau) =(1,-1) .
\end{gather*}
From (\ref{ODE}) we see that $\mu_{0}(c) $ satisf\/ies the ODE
\begin{gather*}
\vartheta ( \vartheta+\beta ) \mu_{0}-c ( \vartheta+\alpha
_{1} ) ( \vartheta+\alpha_{2} ) \mu_{0}=0, 
\end{gather*}
which implies
\begin{gather*}
(1-c) \mu_{2}=\allowbreak\alpha_{1}\alpha_{2}c\mu_{0}+\big[
 ( \alpha_{1}+\alpha_{2} ) c-\beta\big] \mu_{1},
\end{gather*}
and therefore
\begin{gather*}
\sigma_{0}(c) =\alpha_{1}\alpha_{2}c,\qquad\sigma_{1} (
c ) = ( \alpha_{1}+\alpha_{2} ) c-\beta.
\end{gather*}
Proposition~\ref{main} gives
\begin{gather*}
\mu_{n}(c) =(1-c) ^{-n} \vec{P}_{n}(c) \cdot\vec{\mu}(c) ,
\end{gather*}
with $\vec{P}_{n}(c) $ def\/ined by~(\ref{P0}) and
\begin{gather*}
 \vec{P}_{n+1}=c(1-c) \frac{d \vec{P}
_{n}}{dc}+
\begin{bmatrix}
nc & \alpha_{1}\alpha_{2}(1-c) c\\
 1-c & ( \alpha_{1}+\alpha_{2}+n ) c-\beta
\end{bmatrix}
\vec{P}_{n}.
\end{gather*}

\subsection{Hahn polynomials}\label{section3.7}

The Hahn polynomials \cite[(18.20.5)]{MR2723248}
\begin{gather*}
H_{n}(x;\alpha,\beta,N) ={}_{3}F_{2}\left(
\begin{matrix}
-n,-x,n+\alpha+\beta+1\\
-N,\alpha+1
\end{matrix}
;1\right)
\end{gather*}
are orthogonal with respect to the weight function \cite[(18.19)]{MR2723248}
\begin{gather*}
\rho_{H}(x;\alpha,\beta,N) =\frac{(\alpha+1)_{x}}{x!}\frac{(\beta+1)_{N-x}}{(N-x)!},\qquad
\alpha,\beta\notin[-N,-1] ,\qquad N\in\mathbb{N},
\end{gather*}
or
\begin{gather*}
\rho_{H}(x;\alpha,\beta,N) =\binom{\alpha+x}{x}\binom{\beta
+N-x}{N-x}.
\end{gather*}
The relation between $\rho_{H}(x;\alpha,\beta,N) $ and the
weight function of the generalized Hahn polyno\-mials~(\ref{rho gen hahn}) is
given by
\begin{gather*}
\rho_{H}(x;\alpha,\beta,N) =\frac{(\beta+1)_{N}}{N!}\rho ( x;\alpha+1,-N,-N-\beta-1;1 ) .
\end{gather*}

From (\ref{mun gen}), we get (see also \cite[equation~(41)]{MR3286553})
\begin{gather*}
\mu_{n}^{H}=
\sum\limits_{k=0}^{n}
\genfrac{\{}{\}}{0pt}{}{n}{k}
(\alpha+1)_{k} \frac{( \alpha+\beta+2+k)_{N-k}}{( N-k) !},
\end{gather*}
where we have used
\begin{gather*}
\mu_{0}^{H}=\frac{(\alpha+\beta+2)_{N}}{N!}.
\end{gather*}

To obtain the generating function of $\mu_{n}^{H}$, we can use~(\ref{gen Stirling}), and we get
\begin{gather*}
G_{\mu^{H}}(z)=\frac{( \alpha+\beta+2)_{N}}{N!}\,{}_{2}
F_{1}\left[
\begin{matrix}
-N,\ \alpha+1\\
\alpha+\beta+2
\end{matrix}
;1-e^{z}\right] .
\end{gather*}
This generating function seems not to have been considered before.

\begin{remark}
Except for the Bell polynomials (\ref{bell}), all the other families
$\vec{P}_{n}(c) $ associated to the moments $\mu
_{n}(c) $ seem to be new.
\end{remark}

\begin{remark}
Stieltjes transforms of the generalized Krawtchouk and generalized Hahn
polynomials of type~I have been omitted because of the complexity of the formulas.
\end{remark}

\section{Conclusion}\label{section4}

We have developed a technique for computing the moments of weight functions of
hypergeometric type. We have shown that the moments are linear combinations of
the f\/irst $\xi+1$ moments with polynomial coef\/f\/icients in the parameter~$c$.
We have also constructed generating functions for both the moments and the
polynomials associated with them.

All the results in Sections~\ref{section3.3}--\ref{section3.5} are new, and give ef\/f\/icient ways of
computing the moments of the discrete semiclassical polynomials of class~$1$.
The same method can be used to f\/ind the moments of polynomials of class~$s>1$.

In a previous work \cite{MR3277948}, we found the asymptotic zero distribution
of polynomial families satisfying f\/irst-order dif\/ferential-recurrence
relations of the form~(\ref{eq:diff2}). It would be interesting to know if our
results could be extended to include the polynomials $\vec{P}_{n}(c)$ studied in this paper.

\subsection*{Acknowledgements}

We want to thank the anonymous referees and the editors of this
special issue for providing us with extremely valuable suggestions that helped
us to greatly improve our f\/irst draft of the paper.

\pdfbookmark[1]{References}{ref}
\LastPageEnding

\end{document}